\documentclass[UKenglish]{article}
\usepackage{microtype}
\usepackage{macros,stmaryrd,hyperref,tabularx,etex,amsthm, amsmath,amssymb}
\usepackage[inline]{enumitem}

\newtheorem{theorem}{Theorem}
\newtheorem{proposition}[theorem]{Proposition}
\newtheorem{lemma}[theorem]{Lemma}
\newtheorem{definition}[theorem]{Definition}
\newtheorem{corollary}[theorem]{Corollary}

\newcommand{\david}[1]{}
\newcommand{\joseph}[1]{}

\usepackage{authblk}

\newcommand{\fullproof}[1]{{#1}}
\newcommand{\shortproof}[1]{}

\makeatletter
\newcommand*{\ineq}[2][]{%
  \begingroup
    \refstepcounter{equation}%
    \ifx\\#1\\%
    \else
      \label{#1}%
    \fi
    \relpenalty=10000 %
    \binoppenalty=10000 %
    \@eqnnum \ \ensuremath{%
      #2%
    }%
  \endgroup
}
\makeatother

\newcommand{\Var}{{\mathbb P}}
\newcommand{\U}{U}
\newcommand{\fullmod}[2]{{#1}=(W_{#2},{\peq}_{#2},V_{#2})}
\newcommand{\fulllab}[2]{{#1}=(W_{#2},{\peq}_{#2},\labfun_{#2})}

\newcommand{\ur}[1]{{#1}_w}
\newcommand{\T}{{\mathcal T}}
\newcommand{\G }{{\mathcal G}}
\newcommand{\M }{{\mathcal M}}
\newcommand{\N }{{\mathcal N}}
\newcommand{\iM }{{\mathcal M}}

\newcommand{\newmodel}[1]{{#1}^{\rm mod}}
\newcommand{\A}{{\mathcal A}}
\newcommand{\iA}{{\mathcal A}}
\newcommand{\B }{{\mathcal B}}
\newcommand{\iB }{{\mathcal B}}
\newcommand{\dpt}[1]{{\rm dpt}(#1)}
\newcommand{\degp}[1]{\text{lvl}(#1)}
\newcommand{\Labels}{\Lambda}
\newcommand{\simvar}{\sigma}
\newcommand{\srone}{\rho}
\newcommand{\srtwo}{\iota}
\newcommand{\ignore}[1]{}
\newcommand{\peq}{\preccurlyeq}
\newcommand{\labfun}{\lambda}
\newcommand{\sfun}{\msuccfct}
\newcommand{\iltl}{\ensuremath{\sf ITL^e}}
\newcommand{\itlb}{\ensuremath{\sf ITL^p}}
\newcommand{\lang}{{\mathcal L}}
\newcommand{\rng}{{\rm rng}}
\newcommand{\imm}{\mathrel\unlhd}
\newcommand{\cond}{\ll}
\newcommand{\simm}{\mathrel\triangleq}
\newcommand{\btree}[2][\Labels]{{\widehat{\G}} ^{#1}_{#2}}

\newcommand{\sff}[1][\varphi_0]{\text{SF}\!\left(#1\right)}

\newcommand{\wsf}[1]{\Sigma(#1)}

\newcommand{\Edefect}{\mathcal{D}} 
\newcommand{\Qbase}[1][]{}  
\newcommand{\Qstrat}[1][]{^{{\str}\ifx#1\relax \relax\else, #1\fi}}
\newcommand{\str}{{\rm e }} 

\newcommand{\bgraph}[2][\Labels]{\G ^{#1}_{#2}}
\newcommand{\bgraphbase}[2][\Labels]{\bgraph[#1]{#2} = %
  \left(W^{#1}_{#2}, \isuccrel^{#1}_{#2}, \labfun^{#1}_{#2}\right)}

\usepackage{xypic}
\usepackage{graphicx}
\usepackage{multicol}
\def\lang{{\cal L}}
\newcommand{\eqdef}{\overset{def}{=}}

\begin{document}

\title{A Decidable Intuitionistic Temporal Logic}

\author[1]{Joseph Boudou}
\affil[1]{Institut de Recherche en Informatique de Toulouse, Toulouse University, France}
\author[1]{Mart\'in Di\'eguez}
\author[1]{David Fern\'andez-Duque}

\maketitle

\begin{abstract}
We introduce the logic $\iltl$, an intuitionistic temporal logic
based on structures $(W,\peq,\msuccfct)$, where $\peq$ is used to interpret intuitionistic implication and $\msuccfct$ is a $\peq$-monotone function used to interpret temporal modalities.
Our main result is that the satisfiability and validity problems for $\iltl$ are decidable.
We prove this by showing that the logic enjoys the strong finite model property.
In contrast, we also consider a `persistent' version of the logic, $\itlb$, whose models are similar to Cartesian products.
We prove that, unlike $\iltl$, $\itlb$ does not have the finite model property.
\end{abstract}

\section{Introduction}

Intuitionistic logic~\cite{DVD86,MInt}
and its modal extensions~\cite{Ewald,PS86,Simpson94}
play a crucial role in the area of computer science and artificial intelligence. For instance, Pearce's Equilibrium Logic~\cite{Pea96}, which characterises the Answer Set semantics~\cite{MT99,Niemela99} of logic programs (ASP), is defined in terms of the intermediate logic of Here and There~\cite{Hey30}, together with a minimisation criterion. Extensions of Here and There logic allowed the ASP paradigm, already used in a wide range of domains~\cite{FarinasDieguez,BV+2011,GG+10,I13,NBGWB01}, to be applied to reasoning about temporal or epistemic scenarios~\cite{CP07,CerroHS15} while satisfying the theorem of strong equivalence~\cite{CabalarD14,LPV01,CerroHS15}, central to logic programming and nonmonotonic reasoning.

Such modal extensions of Here and There logic are simple cases of a {\em modal intuitionistic logic;} in general, the study of such logics can be a challenging enterprise~\cite{Simpson94}. In particular, there is a huge gap that must be filled regarding combinations of intuitionistic and linear time temporal logic.
Nevertheless, there have been several efforts in this direction, including logics with `past' and `future' tenses \cite{Ewald}
or with `next' $\tnext$, 'eventually' $\diam$ and/or `henceforth' $\ubox$ modalities.
%
The main contributions to the field include the following:
\begin{itemize}\itemsep0pt	
	\item Davies' intuitionistic temporal logic with $\tnext$ \cite{Davies96} was provided Kripke semantics and a complete deductive system by Kojima and Igarashi \cite{KojimaNext}.	
	\item Logics with $\tnext,\ubox$ were axiomatized by Kamide and Wansing \cite{KamideBounded}, where $\ubox$ was interpreted over bounded time.	
	\item Nishimura \cite{NishimuraConstructivePDL} provided a sound and complete axiomatization for an intuitionistic variant of the propositional dynamic logic $\sf PDL$.	
  \item Balbiani and Di\'eguez~\cite{BalbianiDieguezJelia} axiomatized the Here and There variant of $\sf LTL$ with $\tnext,\diam,\ubox$.	
	\item Davoren \cite{Davoren2009} introduced topological semantics for temporal logics and Fern\'an\-dez-Duque \cite{DFD2016} proved the decidability of a logic with $\tnext,\diam$ and a universal modality based on topological semantics.		
\end{itemize}

With the exception of \cite{Davoren2009,DFD2016}, semantics for intuitionistic $\sf LTL$ use frames
of the form $(W,{\irel},{\msuccfct})$, where $\irel$ is a partial order used to interpret the
intuitionistic implication and $\msuccfct$ is a binary relation used to interpret temporal
operators.
Since we are interested in linear time, we will restrict our attention to the case where $\msuccfct$ is a function. Thus, for example, $\tnext p$ is true on some world $w\in W$ whenever $p$ is true on $\msuccfct(w)$.
Note, however, that $\msuccfct$ cannot be an arbitrary function. Intuitionistic semantics have the feature
that, for any formula $\varphi$ and worlds $w\irel v\in W$, if $\varphi$ is true on $w$ then it must
also be true of $v$; that is, truth is {\em monotone.} If we want this property to be preserved by
formulas involving $\tnext$, we need for $\peq$ and $\msuccfct$ to satisfy certain confluence
properties. In the literature, one generally considers frames satisfying
\begin{enumerate*}
\item \label{ItForCon}  $w\irel v$ implies $\msuccfct(w)\irel\msuccfct(v)$ ({\em forward confluence,} or simply {\em confluence}), and
\item\label{ItBackCon}  if $u\succcurlyeq\msuccfct(w)$, there is $v\succcurlyeq w$ such that $\msuccfct(v)=u$ ({\em backward confluence}).
\end{enumerate*}
We will call frames satisfying these conditions {\em persistent frames} (see Sec.~\ref{SecExpPer}),
mainly due to the fact that they are closely related to (persistent) products of modal logics
\cite{mdml}.
Persistent frames for intuitionistic $\sf LTL$ are the frames of the modal logic ${\sf S4}\times {\sf LTL}$, which is non-axiomatizable. For this reason, it may not be surprising that it is unknown whether the intuitionistic temporal logic of persistent frames, which we denote $\itlb$, is decidable.

However, as we will see in Proposition \ref{PropIntCond}, only forward confluence is needed for truth of all formulas to be monotone, even in the presence of $\diam$ and $\ubox$. The frames satisfying this condition are, instead, related to {\em expanding} products of modal logics \cite{pml}, which are often decidable even when the corresponding product is non-axiomatizable. This suggests that dropping the backwards confluence could also lead to a more manageable intuitionistic temporal logic. This logic, which we denote $\iltl$, is the focus of the present paper and, as we will prove in this paper, it enjoys a crucial advantage over $\itlb$: $\iltl$ has the strong finite model property (hence, it is decidable), but $\itlb$ does not. In fact, to the best of our knowledge, $\iltl$ is the first known decidable intuitionistic temporal logic that
\begin{enumerate*}
	\item is conservative over propositional intuitionistic logic,
	\item includes (or can define) the three modalities $\tnext,\diam,\ubox$, and
	\item is interpreted over infinite time.
\end{enumerate*}

\section{Syntax and semantics} \label{SecSynSem}

We will work in the language $\lang$ of $\sf LTL$ given by the following grammar:
\begin{equation*}
  \varphi, \psi \gramdef p \alt \bot \alt
    \varphi \wedge \psi \alt \varphi \vee \psi \alt \varphi \imp \psi \alt
    \tnext \varphi \alt \diam\varphi \alt \ubox\varphi,
\end{equation*}
where $p$ is an element of a countable set of propositional variables $\Var$. Given any formula $\varphi$,
we write $\sff[\varphi]$ for the set of subformulas of $\varphi$ 
and $\nos{\varphi}$ for the cardinality of $\sff[\varphi]$.

A {\em dynamic poset} is a tuple $(W,\peq,\sfun)$, where 
$W$ is a non-empty set of states,
$\irel$ is a partial order,
and
$\sfun$ is a function from $W$ to $W$ that satisfies the following \emph{(forward) confluence} condition:
\begin{equation}\label{EqConf}
\text{for all $w, v \in W$, if $w \irel v$ then $\msuccfct(w) \irel \msuccfct(v)$.}
\end{equation}
An {\em intuitionistic dynamic model,} or simply a {\em model,} is a tuple $\modelbase$ consisting of a dynamic poset equipped
with a valuation function $V$ from $W$ to sets of propositional variables satifying
the \emph{monotonicity} condition: 
\begin{equation}\label{EqVMon}
\text{for all $w, v \in W$, if $w \irel v$ then $V(w) \subseteq V(v)$.}
\end{equation}
In the standard way, we define $\msuccfct[^0](w) = w$ and,
for all $k > 0$, $\msuccfct[^k](w) = \msuccfct\left(\msuccfct[^{k-1}](w)\right)$. Then we define the satisfaction relation $\models$ inductively by:
\begin{alignat*}{4}
  \model, w &\sat p &~&\text{iff } p \in V(w) \\
  \model, w &\sat \tnext \varphi &~&\text{iff } \model, \msuccfct(w) \sat \varphi \\
  \model, w &\sat \bot &&\text{never} \\
  \model, w &\sat \diam \varphi &&\text{iff } \exists k \text{ s.t. } \model, \msuccfct[^k](w) \sat \varphi \\
  \model, w &\sat \varphi \wedge \psi &&\text{iff } 
    \model, w \sat \varphi \text{ and } \model, w \sat \psi ~\\
  \model, w &\sat \ubox \varphi &&\text{iff } \forall k ,~ \model, \msuccfct[^k](w) \sat \varphi \\
  \model, w &\sat \varphi \vee \psi &&\text{iff } \model, w \sat \varphi \text{ or } \model, w \sat \psi \\
  \model, w &\sat \varphi \imp \psi &&\text{iff }
    \forall v \succcurlyeq w,~\text{if } \model, v \sat \varphi \text{ then }\model, v \sat \psi
\end{alignat*}
Given a model $\modelbase$,
a set $\Sigma$ of formulas, and $w\in W$,
we write $\Sigma_\model(w)$ for the set $\ens{\psi \in \Sigma}{\model, w \sat \psi}$; the subscript `$\model$' is omitted when it is clear from the context.
An \emph{eventuality} in $\model$ is a pair $(w, \varphi)$, where $w \in W$ 
and $\varphi$ is a formula such that
either $\varphi = \diam\psi$ for some formula $\psi$ and $\model, w \sat \varphi$,
or $\varphi = \ubox\psi$ for some formula $\psi$ and $\model, w \nsat \varphi$.
The \emph{fulfillment} of an eventuality $(w, \varphi)$ is
the finite sequence $v_0 \ldots v_n$ of states of the model such that
\begin{enumerate*}
  \item for all $k \le n$, $v_0 = \msuccfct[^k](w)$,
  \item if $\varphi = \diam\psi$ then $\model, v_n \sat \psi$ and
        for all $k < n$, $\model, v_k \nsat \psi$, and
  \item if $\varphi = \ubox\psi$ then $\model, v_n \nsat \psi$ and
        for all $k < n$, $\model, v_k \sat \psi$.
\end{enumerate*}

A formula $\varphi$ is {\em satisfiable over a class $\Omega$ of models} if
there is a model $\model\in \Omega$ and a world~$w$ so that $\model,w\sat\varphi$, and {\em valid over $\Omega$} if, for every world $w$ of every model $\model\in \Omega$, $\model,w\sat\varphi$.
Satisfiability (resp. validity) over the class of all intuitionisitic dynamic models
is called {\em satisfiability} (resp. \emph{validity}) for
the {\em expanding domain intuitionisitic temporal logic}~$\iltl$.
We will justify this terminology in the next section. First, we remark that dynamic posets impose the minimal conditions on $\msuccfct$ and $\irel$ in order to preserve the upwards-closure of valuations of formulas. Below, we will use the notation $\llbracket\varphi\rrbracket=\{w\in W \mid \model,w\sat \varphi \}.$

\begin{proposition}\label{PropIntCond}
Let $\mathcal D=(W,{\irel},\msuccfct)$, where $(W,{\irel})$ is a poset and $\msuccfct\colon W\to W$ is any function. Then, the following are equivalent:
\begin{enumerate}

\item\label{ItExpOne} $S$ satisfies the confluence property \eqref{EqConf};

\item\label{ItExpTwo} for every valuation $V$ on $W$ and every formula $\varphi$, $\llbracket \varphi \rrbracket$ is upwards-closed under $\irel$.

\end{enumerate}
\end{proposition}

\begin{proof}
That \ref{ItExpOne} implies \ref{ItExpTwo} follows by a standard structural induction on $\varphi$. The case where $\varphi\in \mathbb P$ follows from the condition on $V$ and most inductive steps are routine. Consider the case where $\varphi=\ubox \psi$, and suppose that $w\irel v$ and $w\in\llbracket \varphi \rrbracket$. Then, for all $i\in\mathbb N$, $\model, \msuccfct[^i](w) \sat \psi$. Since $S$ is confluent, an easy induction shows that, for all $i\in\mathbb N$, $\msuccfct[^i](w)\irel \msuccfct[^i](v)$. Therefore, from the induction hypothesis we obtain that $ \model, \msuccfct[^i](v) \sat \psi$ for all $i$, hence $v\in \llbracket\varphi\rrbracket$. Other cases are similar or easier.

Now we prove that \ref{ItExpTwo} implies \ref{ItExpOne} by contrapositive. Suppose that $(W,{\irel},{\msuccfct})$ does {\em not} satisfy \eqref{EqConf}, so that there are $w\irel v$ such that $\msuccfct(w)\not\irel \msuccfct (v)$. Choose $p\in \mathbb P$ and define $V(u)=\{p\}$ if $w\irel u$, $V(u)=\varnothing$ otherwise. It is easy to see that $V$ satisfies the monotonicity condition \eqref{EqVMon}. But, $p\not\in V(v)$, from which it follows that $(\mathcal D,V),w\sat \tnext p$ but $(\mathcal D,V),v\nsat \tnext p$.
\end{proof}

We are concerned with the satisfiability and validity problems for $\iltl$. Observe that satisfiability in propositional intuitionistic logic is equivalent to satisfiability in classical propositional logic. This is because, if $\varphi$ is classically satisfiable, it is trivially intuitionistically satisfiable in a one-world model; conversely, if $\varphi$ is intuitionistically satisfiable, it is satisfiable in a finite model, hence in a maximal world of that finite model, and the generated submodel of a maximal world is a classical model. Thus it may be surprising that the same is not the case for intuitionistic temporal logic:

\begin{lemma}
Any formula $\varphi$ of the temporal language that is classically satisfiable is satisfiable in a dynamic poset. However, there is a formula satisfiable on a dynamic poset that is not classically satisfiable.
\end{lemma}

\proof
If $\varphi$ is satisfied on a classical model $\model$, then we may regard $\model$ as an intuitionistic model by letting $\peq$ be the identity.
On the other hand, consider the formula $\neg\tnext p\wedge \neg\tnext\neg p$.
Classically, this formula is equivalent to $\neg\tnext p\wedge \tnext p$, and hence unsatisfiable. Define a model $\modelbase$, where $W=\{w,v,u\}$, $x\peq  y$ if $x=y$ or $x=v$, $y=u$, $\sfun (w)=v$ and $\sfun (x)=x$ otherwise, and $V(u)=\{ p \}$.
Then, one can check that $\model, w \sat \neg\tnext p\wedge \neg\tnext\neg p$.
\endproof

Hence the decidability of the intuitionistic satisfiability problem is not a corollary of the
classical case.
In Section~\ref{SecDec}, we will prove that both the satisfiability and the validity problems are
decidable.

\section{Expanding and persistent frames} \label{SecExpPer}

In this section, we discuss expanding and persistent models, and compare them to dynamic models as we have defined above.

\subsection{Expanding model property}\label{SubsStrat}

The logic $\iltl$ is closely related to {\em expanding products} of modal logics \cite{pml}.
In this subsection,
we introduce stratified and expanding frames,
and show that satisfiability and validity on arbitrary models is equivalent to satisfiability and validity on expanding models.
To do this,
it is convenient to represent posets using acyclic graphs.

\begin{definition}
A \emph{directed acyclic graph} is a tuple $(W, \isuccrel)$, where
$W$ is a set of vertices,
${\isuccrel} \subseteq W \times W$ is a set of edges
whose reflexive, transitive closure $\isuccrel^*$ is antisymmetric. We will tacitly identify $(W, \isuccrel)$ with the poset $(W, \isuccrel^*)$. A \emph{path} from $w_1$ to $w_2$
is a finite sequence $v_0 \ldots v_n \in W$ such that $v_0 = w_1$, $v_n = w_2$
and for all $k < n$, $v_k \isuccrel v_{k+1}$. 
A \emph{tree} is an acyclic graph $(W, \isuccrel )$
with an element $r \in W$, called the root, such that for all $w \in W$
there is a unique path from $r$ to $w$.
A poset $(W,\peq)$ is also a tree if there is a relation $\isuccrel$ on $W\times W$ such that
$(W,\isuccrel)$ is a tree and ${\peq}={\isuccrel^*}$.
\end{definition}

\begin{definition} \label{def:stratified}
  A model $\modelbase$ is \emph{stratified} if there is a partition
  $\left\{W_n\right\}_{n < \omega}$ of $W$ such that
  \begin{enumerate*}
    \item each $W_n$ is closed under $\irel$,
          \label{cond:stratified:closed}
    \item for all $n$, there is relation $\isuccrel_n$ such that
          $(W_n, \reduc{\irel}{W_n})$ is a labeled tree, and
          \label{cond:stratified:tree}
    \item if $w \in W_n$ then $\msuccfct(w) \in W_{n+1}$.
          \label{cond:stratified:succ}
  \end{enumerate*}
  If $\model$ is stratified, we write $\peq_n,\msuccfct_n,$ and $V_n$ instead of $\reduc{\irel}{W_n},\reduc{\msuccfct}{W_n},$ and $\reduc{V}{W_n}$ and write $\model_n=(W_n,\irel_n,V_n)$. If moreover we have that $\msuccfct(w)\peq \msuccfct(v)$ implies $w\peq v$, then we say that $\model$ is an {\em expanding model.}
\end{definition}

Given a finite, non-empty set of formulas $\Sigma$ closed under subformulas,
a model $\modelbase[\Qbase]$, and a state $w\Qbase \in W\Qbase$,
we will construct a stratified model $\modelbase[\Qstrat]$
such that 
for the root $w\Qstrat$ of $W\Qstrat_{0}$, $\Sigma({w\Qstrat}) = \Sigma({w\Qbase})$.
To this end,
we first define the set $\Edefect = \nat \times \nat \times 2^\Sigma$ of possible defects.
Since $\Sigma$ is finite and not empty, we assume that $\Edefect$ is ordered such that
for each $k \in \mathbb N$, the $k$\textsuperscript{th} element $(x, y, S)$ of $\Edefect$
is such that $x \leq k$.
Then, for each $k \in \nat$, we construct inductively a tuple $(\U_k, \isuccrel_k, h_k)$
where $\U_k \subseteq \nat \times \nat$, ${\isuccrel_k} \subseteq \U_k \times \U_k$ and
$\deffun{h_k}{\U_k}{W\Qbase}$.
The model $\model\Qstrat$ is defined from all these tuples
and the whole construction proceeds as follows:

\subparagraph*{Base case.} Let $\U_0 = \{0\} \times \nat$, $\isuccrel_0 = \emptyset$ and 
$h_0$ be such that for all $(0,y) \in \U_0$, $h_0(0,y) = {\msuccfct[\Qbase]}^y\left(w\Qbase\right)$.

\subparagraph*{Inductive case.} Let $k > 0$ and suppose that $(\U_k, \isuccrel_k, h_k)$ has already been
constructed.
Let $(x, y, S)$ be the $k$\textsuperscript{th} element of $\Edefect$.
If
\begin{enumerate*}[label=(D\arabic*)]
  \item $ (x, y) \in \U_k$, \label{cond:stratified:ind:belong}
  \item $\Sigma(h_k(x,y)) \neq S$, and
  \item there is $ v \in W\Qbase$ such that $h_k(x,y) \irel\Qbase v$ and $\wsf{v} = S$,
    \label{cond:stratified:ind:exist}
\end{enumerate*}
then we construct $(\U_{k+1}, \isuccrel_{k+1}, h_{k+1})$ such that:
\begin{align*}
  \U_{k+1} &= \U_k \cup \ens{(c,d) \in \nat \times \nat}{c = k+1 \text{ and } d \ge y} \\
  \isuccrel_{k+1} &= \mathord{\isuccrel_k} \cup
    \ens{((a,b),(c,d))}{a = x, c = k+1, d \ge y \text{ and } b = d} \\
    h_{k+1} &= h_k \cup \ens{((c,d),w)}{c = k+1, d \ge y \text{ and } w = {\msuccfct[\Qbase]}^{d-y}(v)}
\end{align*}
Otherwise $(\U_{k+1}, \isuccrel_{k+1}, h_{k+1}) = (\U_k, \isuccrel_k, h_k)$.

\subparagraph*{Final step.} We construct $\modelbase[\Qstrat]$ such that $W\Qstrat  = \bigcup_{k \in \nat} \U_k$, ${\irel\Qstrat} = \mathord{(\isuccrel\Qstrat)^*},$ where $ \mathord{\isuccrel\Qstrat} = \bigcup_{k \in \nat} \mathord{\isuccrel_k},$
\[\msuccfct\Qstrat = \ens{((a,b),(c,d)) \in W\Qstrat \times W\Qstrat}{a = c \text{ and } d = b + 1},\]
and $  V\Qstrat(x,y) = V\Qbase\left(h_x(x,y)\right).
$

\begin{lemma} \label{lem:stratified:aux}
  For all $(x,y),(x',y') \in W\Qstrat$,
  if $(x,y) \irel\Qstrat (x',y')$,
  then $x \le x'$, $y = y'$ and $h_{x}(x,y) \irel\Qbase h_{x'}(x',y')$.
\end{lemma}
\begin{proof}
  Suppose that $(x,y) \irel\Qstrat (x',y')$.
  There is a sequence $(x_0, y_0) \ldots (x_n, y_n)$ such that
  $(x_0, y_0) = (x, y)$, $(x_n, y_n) = (x', y')$ and
  for all $i < n$, $(x_i, y_i) \isuccrel\Qstrat (x_{i+1}, y_{i+1})$.
  By construction, for all $i < n$, $(x_i, y_i) \isuccrel_{x_{i+1}} (x_{i+1}, y_{i+1})$
  and $y_i = y_{i+1}$.
  \newcommand{\Qdef}{''_{i+1}}
  Let $(x\Qdef, y\Qdef, S\Qdef)$ be the $(x_{i+1}-1)$\textsuperscript{th} element of $\Edefect$
  and $v\Qdef$ the element of $W\Qbase$ choosen at the $(x_{i+1}-1)$\textsuperscript{th} step.
  By construction, $x\Qdef = x_i$, $y\Qdef \le y_i$
  and by the ordering of $\Edefect$, $x_i \le x_{i+1} - 1$.
  Moreover, $h_{x_i}(x, y\Qdef) \irel\Qbase v\Qdef$.
  Since
  $h_{x_i}(x_i, y_i) = \msuccfct\Qbase^{y_i - y\Qdef}\left(h_{x_i}\left(x_i, y\Qdef\right)\right)$
  and
  $h_{x_{i+1}}(x_{i+1}, y_{i+1}) = \msuccfct\Qbase^{y_{i+1} - y\Qdef}\left(v\Qdef\right)$,
  by the confluence condition for $\model\Qbase$,
  $h_{x_i}(x_i, y_i) \irel\Qbase h_{x_{i+1}}(x_{i+1}, y_{i+1})$.%
  \joseph{if space is needed, I can shrink this proof.}
\end{proof}

\begin{lemma}
  $\model\Qstrat$ is an expanding model.
\end{lemma}
\begin{proof}
  First we check that $\model\Qstrat$ is stratified.
  By Lemma~\ref{lem:stratified:aux}, $\irel\Qstrat$ is antisymetric, hence a partial order.
  For the monotonicity condition, suppose that $(x,y) \irel\Qstrat (x',y')$.
  By Lemma~\ref{lem:stratified:aux}, $h_x(x,y) \irel\Qbase h_{x'}(x',y')$ and
  by the monotonicity condition for $\model\Qbase$, 
  $V\Qbase\left(h_x(x,y)\right) \subseteq V\Qbase\left(h_{x'}(x',y')\right)$.
  For the confluence condition, it suffices to observe that by construction,
  if $(x,y) \isuccrel\Qstrat (x', y')$ then $(x, y+1) \isuccrel\Qstrat (x', y'+1)$.
  Therefore, $\model\Qstrat$ is a model.
  To prove that $\model\Qstrat$ is stratified,
  define $W\Qstrat_n = \ens{(x,y) \in W\Qstrat}{y = n}$
  for all $n \in \nat$.
  Conditions~\ref{cond:stratified:succ} of Def.~\ref{def:stratified} trivially holds and
  condition~\ref{cond:stratified:closed} comes directly from Lemma~\ref{lem:stratified:aux}.
  To prove condition~\ref{cond:stratified:tree}, it suffices to observe that by construction,
  for all $(x,y) \in W\Qstrat$, either $x = 0$ or
  there is exactly one state $(x',y') \in W\Qstrat$ such that $(x', y') \isuccrel\Qstrat (x,y)$.
  Therefore, by Lemma~\ref{lem:stratified:aux}, for all $(x,y) \in W\Qstrat$,
  there is a unique path from $(0, y)$ to $(x, y)$.
  Finally, to prove that $\model\Qstrat$ is expanding,
  suppose that $(c,b) \in W\Qstrat$ and $(a, b+1) \isuccrel\Qstrat (c, b+1)$.
  Then the $(c-1)$\textsuperscript{th} element of $\Edefect$ is $(a, y, S)$ for some $y, S$.
  Moreover, since $(c,b) \in W\Qstrat$, $b \ge y$ and since $(x,y) \in W\Qstrat$,
  $(a,b) \in W\Qstrat$ and $(a,b) \isuccrel\Qstrat (c,b)$.
  Therefore it can easily be proved
  by induction on the length of the path from $\msuccfct\Qstrat(w)$ to $\msuccfct\Qstrat(v)$
  that $\msuccfct\Qstrat(w) \irel\Qstrat \msuccfct\Qstrat(v)$ implies $w \irel\Qstrat v$.
\end{proof}

\begin{lemma}
  For any state $(x,y) \in W\Qstrat$ and any $\psi \in \Sigma$, 
  $\model\Qstrat, (x,y) \sat \psi$ if and only if $\model\Qbase, h_x(x,y) \sat \psi$.
\end{lemma}
\begin{proof}
  The proof is by induction on the size $\nos{\psi}$ of the formula.
  The cases for propositional variables, falsum, conjunctions and disjunctions are straightforward.
  For the temporal modalities, it suffices to observe that
  for all $(x, y) \in W\Qstrat$ and all $n \in \nat$,
  $(x, y + n) \in W\Qstrat$ and $h_x(x, y + n) = \msuccfct\Qbase^n\left(h_x(x,y)\right)$.
  Finally, for implication, suppose first that $\model\Qstrat, (x,y) \nsat \psi_1 \imp \psi_2$.
  Then there is $(x',y')$ such that $(x, y) \irel\Qstrat (x', y')$,
  $\model\Qstrat, (x', y') \sat \psi_1$ and $\model\Qstrat, (x', y') \nsat \psi_2$.
  By Lemma~\ref{lem:stratified:aux}, $h_x(x, y) \irel\Qbase h_{x'}(x', y')$
  and by induction hypothesis, $\model\Qbase, h_{x'}(x', y') \sat \psi_1$ and
  $\model\Qbase, h_{x'}(x', y') \nsat \psi_2$.
  Therefore, $\model\Qbase, h_x(x, y) \nsat \psi_1 \imp \psi_2$.
  For the other direction suppose that $\model\Qbase, h_x(x,y) \nsat \psi_1 \imp \psi_2$.
  There is $v' \in W\Qbase$ such that $h_x(x, y) \irel\Qbase v'$,
  $\model\Qbase, v' \sat \psi_1$ and $\model\Qbase, v' \nsat \psi_2$.
  Let $k$ be such that $(x, y, \Sigma(v'))$ is the $k$\textsuperscript{th} element of $\Edefect$.
  Condition~\ref{cond:stratified:ind:exist} trivially holds and since $x \le k$,
  condition~\ref{cond:stratified:ind:belong} holds too.
  Hence, there is $(x', y') \in W\Qstrat$ such that $\Sigma(h_{x'}(x', y')) = \wsf{v'}$
  and either $(x', y') = (x, y)$ or $(x', y') \isuccrel\Qstrat (x, y)$.
  By induction hypothesis, $\model\Qstrat, (x',y') \sat \psi_1$ and
  $\model\Qstrat, (x', y') \nsat \psi_2$, hence $\model\Qstrat, (x, y) \nsat \psi_1 \imp \psi_2$.
\end{proof}

In conclusion, we obtain the following:

\begin{theorem}\label{TheoStrat}
  A formula $\varphi$ is satisfiable (resp. falsifiable) on an intuitionistic dynamic model
  if and only if   it is satisfiable (resp. falsifiable) on an expanding model.
\end{theorem}

\subsection{Persistent frames}

Expanding models were introduced as a weakening of product models.
They often lead to logics with a less complex validity problem.
Thus it is natural to also consider a variant of $\iltl$
interpreted over product models,
or over the somewhat wider class of persistent models.

\begin{definition}
Let $(W,{\preccurlyeq})$ be a poset.
If $\msuccfct\colon W\to W$ is such that,
whenever $v \succcurlyeq \sfun(w)$, there is $u\succcurlyeq w$ such that $v=\sfun (u)$,
we say that $\sfun$ is {\em backward confluent}.
If $\sfun$ is both forward and backward confluent, we say that it is {\em persistent}.
A tuple $(W,{\preccurlyeq},\msuccfct)$ where $\msuccfct$ is persistent is a {\em persistent
intuitionistic temporal frame}, and the set of valid formulas over the class of persistent
intuitionistic temporal frames is denoted $\itlb$, or {\em persistent domain $\sf LTL$}.
\end{definition}

The name `persistent' comes from the fact that Theorem \ref{TheoStrat} can be modified to obtain a stratified model $\model'$ where $\msuccfct'\colon W'_k\to W'_{k+1}$ is an isomorphism, i.e.~whose domains are persistent with respect to $\msuccfct'$. As we will see, the finite model property fails over the class of persistent models.

\begin{lemma}\label{LemmInfExam}
The formula $\varphi=\neg\neg\diam\ubox p\to \diam\neg\neg\ubox p$ is not valid over the class of persistent models.
\end{lemma}

\proof
Consider the model $\mathcal M=(W,{\preccurlyeq},\msuccfct,V)$, where $W=\mathbb Z \cup \{r\}$ with $r$ a fresh world not in $\mathbb Z$, $w\preccurlyeq v$ if and only if $w=r$ or $w=v$, $\msuccfct(r)=r$ and $\msuccfct(n)=n+1$ for $n\in\mathbb Z$, and $\llbracket p\rrbracket =[0,\infty)$. It is readily seen that $\mathcal M$ is a persistent model, that $\model,r\sat \neg\neg\diam\ubox p$ (since every maximal world above $r$ satisfies $\diam\ubox p$), yet $\model,r\nsat \diam\neg\neg\ubox p$, since there is no $n$ such that $\model, \msuccfct^n(r) \sat \neg\neg\ubox p$. It follows that $\model,r\nsat  \varphi $, and hence $\varphi$ is not valid, as claimed.
\endproof

\begin{lemma}\label{LemmFinExam}
The formula $\varphi$ (from Lemma \ref{LemmInfExam}) is valid over the class of finite, persistent models.
\end{lemma}

\proof
Let $\model=(W,{\preccurlyeq},\msuccfct,V)$ be a finite, persistent model, and assume that $\model,w\sat \neg\neg\diam\ubox p $.
Let $v_1,\hdots,v_n$ enumerate the maximal elements of $\{v\in W \mid w\peq v\}$. 
For each $i\leq n$, let $k_i$ be large enough so that $\model,\msuccfct^{k_i}(v_i)\sat \ubox p $, and let $k=\max k_i$. We claim that $\model, \msuccfct^k(w)\sat \neg\neg\ubox p $, which concludes the proof. Let $u\succcurlyeq \msuccfct^k(w)$ be any leaf. Then, there is $v_i\succcurlyeq w$ such that $u=\msuccfct^{k}(v_i)$ (since compositions of persistent functions are persistent). But, since $k\geq k_i$, we obtain $\model,u\sat \ubox p$, as desired.
\endproof

The following is then immediate from Lemmas \ref{LemmInfExam} and \ref{LemmFinExam}:

\begin{theorem}
${\itlb}$ does not have the finite model property.
\end{theorem}

Thus our decidability proof for $\iltl$, which proceeds by first establishing a strong finite model property, does not carry over to $\itlb$. Whether $\itlb$ is decidable remains open.

\section{Combinatorics of intuitionistic models}

In this section we introduce some combinatorial tools we will need in order to prove that $\iltl$ has the strong finite model property, and hence is decidable. We begin by discussing labeled structures, which allow for a graph-theoretic approach to intuitionistic models.

\subsection{Labeled structures and quasimodels}

\begin{definition}
Given a set $\Labels$ whose elements we call `labels' and a set $W$, a {\em $\Lambda$-labeling
function} on $W$ is any function $\labfun\colon W\to \Labels$. A structure ${\mathcal
S}=(W,R,\labfun)$ where $W$ is a set, $R\subseteq W\times W$ and $\labfun$ is a labeling function on
$W$ is a {\em $\Labels$-labeled structure,} where `structure' may be replaced with `poset', `directed graph', etc.
\end{definition}

A useful measure of the complexity of a labeled poset is given by its level:

\begin{definition}
Given a labeled poset $\A=(W, {\peq},\labfun)$ and an element $w \in W$,
an \emph{increasing chain} from $w$ of length $n$ is
a sequence $v_1 \ldots v_n$ of elements of $W$ such that $v_1 = w$ and $\forall i < n ,~ v_i \prec v_{i+1},$
where $u\prec u'$ is shorthand for $u\peq u'$ and $u'\not\peq u$. The chain $v_1 \ldots v_n$ is {\em proper} if it moreover satisfies $\forall i < n ,~ \labfun\left(v_i\right) \ne \labfun\left(v_{i+1}\right).$
The {\em depth} $\dpt{w} \in \nat \cup \{\omega\}$ of $w$ is defined such that
$\dpt{w} = m$ if $m$ is the maximal length of all the increasing chains from $w$
and $\degp{w} = \omega$ is there is no such maximum. 
%
Similarly,
the \emph{level} $\degp{w} \in \nat \cup \{\omega\}$ of $w$ is defined such that
$\degp{w} = m$ if $m$ is the maximal length of all the proper increasing chains from $w$
and $\degp{w} = \omega$ if there is no such maximum. 
The level $\degp{\A}$ of $\A$ is the maximal level of all its elements.
\end{definition}

An important class of labeled posets comes from intuitionistic models.

\begin{definition}
Given an intuitionistic Kripke model $\model = (W, \irel, V)$ and
a set $\Sigma$ of intuitionistic formulas closed under subformulas,
it can easily be checked that for all $w, v \in W$,
if $w \irel v$ then $\Sigma(w) \subseteq \Sigma(v)$.
We denote the labeled poset $(W, \irel, \Sigma(\cdot))$ by $\model^\Sigma$.
Conversely,
given a labeled poset $\A=(W,{\peq},\labfun)$ over $2^\Sigma$,
the valuation $V_{\labfun}$ is defined such that
$V_{\labfun}(w) = \{p \in \Var \mid p \in \labfun(w) \}$ for all $w \in W$, and denote the resulting model by $\newmodel\A$.
\end{definition}

\begin{definition}
Let $\Sigma$ be a finite set of formulas closed under subformulas and $\A=(W,{\peq},\labfun)$ be a $2^\Sigma$-labeled poset. We say that $\A$ is a {\em $\Sigma$-quasimodel} if $\labfun $ is monotone in the sense that $w\peq  v$ implies that $\labfun(w)\subseteq \labfun(v)$, and whenever $\varphi\to\psi\in \Sigma$ and $w\in W$, we have that $\varphi\to\psi \in \labfun (w)$ if and only if, for all $v$ such that $w\peq v$, if $\varphi\in\labfun(v)$ then $\psi\in\labfun(v)$.
\end{definition}

\subsection{Simulations, immersions and condensations}

As is well-known, truth in intuitionistic models is preserved by bisimulation, and thus this is usually the appropriate notion of equivalence between different models.
However, for our purposes, it is more convenient to consider a weaker notion, which we call {\em bimersion}.

\begin{definition}\label{DefSim}
Given two labeled posets $\A=(W_\iA,\peq_\iA,\labfun_\iA)$ and $\B =(W_\iB ,\peq_\iB ,\labfun_\iB )$
and a relation $R\subseteq  W_\iA\times W_\iB $,
we write $\dom(R)$ for
\[\ens{w\in W_\iA}{\exists v \in W_\iB  ,~ (w,v) \in R}\]
and $\rng(R)$ for $\ens{v \in W_\iB }{\exists w \in W_\iA ,~ (w,v) \in R}$.
A relation ${\simvar} \subseteq W_\iA \times W_\iB $ is
a \emph{simulation} from $\A$ to $\B $ if
$\dom({\simvar})=W_\iA$ and
whenever $w\mathrel\simvar v$,
it follows that $\labfun_\iA(w)=\labfun_\iB  (v)$,
and if $w\peq_\iA w'$ then there is $v'$ so that $v\peq_\iB  v'$ and $w'\mathrel \simvar v'$.

A simulation is called an {\em immersion} if it is a function.
If an immersion $\simvar\colon W_\iA\to W_\iB $ exists, we write $\A\imm \B $.
If, moreover, there is an immersion $\tau\colon W_\iB \to W_\iA$,
we say that they are {\em bimersive},
write $\A\simm \B $, and call the pair $(\simvar,\tau)$ a {\em bimersion}.
A {\em condensation from $\A$ to $\B$} is a bimersion $({\srone},{\srtwo})$
so that $\srone\colon W_\iA\to W_\iB $, $\srtwo\colon W_\iB \to W_\iA$,
$\srone$ is surjective, and $\srone\srtwo$ is the identity on $W_\iB$.
If such a condensation exists we write $\B \cond \A$.
Observe that $\B \cond \A$ implies that $\B \simm \A$. 

If $\M ,\N $ are models and $\Sigma$ a set of formulas closed under subformulas, we write $\M \imm_\Sigma\N $ if $\M ^\Sigma\imm\N ^\Sigma$, and define $\simm_\Sigma,\cond_\Sigma$ similarly. We may also write e.g.~$\A\cond \M $ if $\A$ is $2^\Sigma$-labeled and $\A\cond\M ^\Sigma$. 
\end{definition}

It will typically be convenient to work with immersions rather than simulations: however, as the next lemma shows, not much generality is lost by this restriction.

\begin{lemma}\label{LemmThereIsFun}
  Let $\fulllab\A\iA$ and $\fulllab\B\iB$ be labeled posets.
  If a simulation ${\simvar}\subseteq W_\iA\times W_\iB$ exists,
  $W_\iA$ is a finite tree, and $w\mathrel\simvar w'$,
  then there is an immersion ${\simvar}'\subseteq W_\iA\times W_\iB$
  such that $w\in \dom({\simvar}')$.
\end{lemma}

\proof
By a straightforward induction on the depth of $w$.
Let $D$ be the set of daughters of $w$,
and for each $v\in D$, choose $v'$ so that $v\mathrel \simvar v'$ and $w'\peq v$.
By the induction hypothesis, there is an immersion $\simvar_{v}'$ with $v\in \dom(\simvar_{v}')$.
Then, one readily checks that $\{(w,w')\}\cup\bigcup_{v\in D}\simvar_{v}'$
is also an immersion, as needed. 
\endproof

Condensations are useful for producing (small) quasimodels out of models.

\begin{proposition} \label{prop:IL}
  Given
  an intuitionistic dynamic model $\fullmod\M \iM $,
  a set $\Sigma$ of intuitionistic formulas that is closed for subformulas,
  and a $2^\Sigma$-labeled poset $\fulllab\A\iA$ over $\Sigma$,
  if $\A \cond \M$,
  then $\A$ is a quasimodel.
\end{proposition}

\fullproof{
\begin{proof}
Let $(\srone,\srtwo)$ be a condensation from $\model^\Sigma$ to $\A$. If $w\peq_\iA v $, then $\srtwo(w)\peq_\iM \srtwo(v)$, so that $\labfun(w)=\Sigma(\srtwo(w))\subseteq \Sigma( \srtwo(v))=\labfun(v)$. Next, suppose that $\varphi\to\psi\in \labfun(w)$, and consider $v$ such that $w\peq_\iA v$. Then, $\M , \srtwo(w) \sat \varphi \imp \psi$. Since $\srtwo$ is an immersion, $\srtwo(w)\peq_\iM \srtwo(v)$, hence if $\M , \srtwo(v) \sat \varphi  $, then also $\M , \srtwo(v) \sat   \psi$. Thus if $\varphi\in\labfun_\iA(v)$, it follows that $\psi\in\labfun_\iA(v)$.
Finally, suppose that $\varphi\to\psi \in \Sigma\setminus \labfun(w)$. Then, $\M,\srtwo(w)\nsat \varphi\to\psi$, so that there is $v \in W_\iA$ such that $\srtwo(w) \irel_\iA v$,
  $\M, v \sat \varphi$ and $\M, v \nsat \psi$.
  It follows that $\varphi \in \labfun(\srone (v))$ and $\psi \not\in \labfun(\srone (v))$, and since $\srone$ is an immersion we also have that $w=\srone\srtwo(w)\peq \srone(v)$, as needed.
\joseph{We can not omit this proof because it is mentionned in the proof of \ref{lem:transformations}.}
\end{proof}
}
\shortproof{
\begin{proof}
See Appendix.
\end{proof}
}

\subsection{Normalized labeled trees}

In order to
count the number of different labeled trees up to bimersion,
we construct,
for any set $\Labels$ of labels and any $k \ge 1$,
the labeled directed acyclic graph $\bgraphbase{k}$ by induction on $k$ as follows.

\subparagraph*{Base case.}
For $k = 1$, let $\bgraphbase 1$ with $W^\Labels_{1} = L$, $\isuccrel^\Labels_{1} = \varnothing$, and $\labfun^\Labels_{1}(w) = w $ for all $ w \in W^\Labels_{1}.$

\subparagraph*{Inductive case.}
Suppose that $\bgraphbase{k}$ has already been defined.The graph $\bgraphbase{k+1}$ is constructed such that:
\begin{align*}
  W^\Labels_{k+1} &= W^\Labels_{k} \cup P \\
  \isuccrel^\Labels_{k+1} &= \isuccrel^\Labels_{k} \cup
    \ens{(x,y) \in W^\Labels_{k+1} \times W^\Labels_{k+1}}{\exists (\ell,S) \in P ,~
    x \in S \text{ and } y = (\ell, S)} \\
  \labfun^\Labels_{k+1}(w) &= \begin{cases}
      \labfun^\Labels_{k}(w) &\text{if } w \in W^\Labels_{k} \\
      \ell       &\text{if } w = (\ell, S) \in P
    \end{cases}
    \end{align*}
where $P = \ens{(\ell, S) \in L \times \parts{W^\Labels_{k}}}{ 
    \forall y \in S,~ \labfun^\Labels_{k}(y) \ne \ell}$.

Note that $\bgraphbase k$ is typically not a tree, but we may unravel it to obtain one.

\begin{definition}
Given a  labeled directed acyclic graph $\G  = (W, \isuccrel, \labfun)$ and a node $w \in W$,
the \emph{unraveling} of $\G $ from $w$ is
the labeled tree $\ur\T  =(\ur W, \ur\isuccrel, \ur\labfun)$ such that
$\ur W$ is the set of all the paths from $w$ in $\G$,
$\xi \mathrel{\ur\isuccrel}\zeta $ if and only if there is $v \in W $
such that $ \zeta = \xi v$, and $\ur\labfun(v_0 \ldots v_n) = \labfun(v_n)$.
\end{definition}

\begin{proposition} \label{PropIsBisim}
  For any rooted labeled tree
  $\T = (W, \isuccrel, \labfun)$ over a set $\Labels$ of labels,
  if the level of $\T$ is finite then
  there is a condensation from $\T$ to an unraveling of $\btree{\degp \T}$.
\end{proposition}
\begin{proof}
  Let $\T = (W, \isuccrel, \labfun)$ be a labeled directed acyclic graph with root $r$.
  We write $\irelneq$
  for the transitive closure of $\isuccrel$
  and $\irel$
  for the reflexive closure of $\irelneq$.
  The proof is by induction on the level $n = \degp \T$ of $\T$.
  For $n = 1$, observe that this means that $\labfun(w)=\labfun(r)$ for all $w\in W$.
  Let $\srone = W\times \{\labfun(r)\}$ and $\srtwo=\{(\labfun(r),r)\}$.
  It can easily be checked that
  $(\srone,\srtwo)$ is a condensation.
  For $n > 1$, suppose the property holds for all rooted labeled trees $\T'$ such that
  $\degp \T' < n$.
  Define the following sets:
  \begin{align*}
    N &= \ens{w \in W}{\labfun(w) \neq \labfun(r)
            \text{ and for all } v \irelneq w ,~ \labfun(v) = \labfun(r)} \\
    S &= \ens{w \in W}{\text{for all } v \irel w ,~ \labfun(v) = \labfun(r)}
  \end{align*}
  Clearly, for all $w \in N$, $\degp w < n$.
  Therefore, by induction, there is a condensation $({\srone}_w,{\srtwo}_w)$
  from the subgraph of $\T$ generated by $w$
  to the unraveling of $\G^\Labels_{n-1}$ from some $y_w \in W^\Labels_{n-1}$.
  Let us define $r' = (\labfun(r), \ens{y_w}{w \in N})$
  and consider the unraveling $\G$ of $\G^\Labels_n$ from $r'$.
  It can easily be checked that
  ${\srone} = \left(S \times \left\{r'\right\} \right) \cup \bigcup_{w \in W} {\srone}_w$
  is an immersion from $\T$ to~$\G$,
  ${\srtwo'} = \{(r',r)\} \cup \bigcup_{w \in W} {\srtwo}_w$
  is a simulation from $\G$ to~$\T$
  and $\srtwo'\subseteq \srone^{-1}$.
  Using Lemma~\ref{LemmThereIsFun}, we can then choose an immersion $\srtwo\subseteq \srtwo'$, so
  that $(\srone,\srtwo)$ is a condensation from $\T$ to~$\G$.
\end{proof}

\newcommand{\nre}[2]{E^{#1}_{#2}}
\newcommand{\nrq}[2]{Q^{#1}_{#2}}

Finally, let us define recursively $\nre{n}{k}$ and $\nrq{n}{k}$ for all $n,k\in\nat$ by:

\begin{equation*}
  \nre{n}{k} = \begin{cases}
    0 &\text{if } k = 0 \\
    \nre{n}{k-1} + n 2^{\nre{n}{k-1}} &\text{otherwise}
  \end{cases}
  \qquad
  \nrq{n}{k} = \begin{cases}
    0 &\text{if } k = 0 \\
    1 + \nre{n}{k-1} \nrq{n}{k-1} &\text{otherwise}
  \end{cases}
\end{equation*}

The following lemma can be proven by a straightforward induction, left to the reader.

\begin{lemma}\label{LemmBoundsGraph}

  For any finite set $\Labels$ with cardinality $n$ and all $k \in \nat$,
\begin{enumerate*}

\item
  the cardinality of $\bgraph{k}$ is bounded by $\nre{n}{k}$, and
  
\item 
  the cardinality of any unraveling of $\bgraph k$
  is bounded by $\nrq{n}{k}$.

\end{enumerate*}

\end{lemma}

From these and Proposition \ref{PropIsBisim}, we obtain the following:

\begin{theorem}\label{TheoKruskal}

\begin{enumerate}

\item 
Given a set of labels $\Labels$ and a $\Labels$-labeled tree $\T$ of level $k<\omega$, there is $\T'\simm \T$ bounded by $\nrq{|\Labels|}{k}$. We call $\T'$ the {\em normalized $\Labels$-labeled tree for $\T$.}

\item \label{ItKruskal} Given a sequence of $\Labels$-labeled trees $\T_1,\hdots,\T_n$ of level $k<\omega$ with $n > \nre{|\Labels|}{k}$, there are indexes $i<j\leq n$ such that $\T_i\simm \T_j$.

\end{enumerate}

\end{theorem}

The second item may be viewed as a finitary variant of Kruskal's theorem for labeled trees \cite{Kruskal1960}. When applied to quasimodels, we obtain the following:

\begin{proposition}\label{PropBound}

Let $\Sigma$ be a set of formulas closed under subformulas with $|\Sigma|=s <\omega$.

\begin{enumerate}

\item 
Given a tree-like $\Sigma$-quasimodel $\T$ and a formula $\varphi$, there is a tree-like $\Sigma$-quasimodel $\T'\simm_{\Sigma}\T$ bounded by $\nrq{2^{s}}{s+1}$. We call $\T'$ the {\em normalized $\Sigma$-quasimodel for $\T$.}

\item Given a sequence of tree-like $\Sigma$-quasimodels $\T_1,\hdots,\T_n$ with $n > \nre{2^{s}}{s+1}$, there are indexes $i<j\leq n$ such that $\T_i\simm\T_j$.

\end{enumerate}
\end{proposition}

\proof
Immediate from Proposition \ref{prop:IL} and Lemma \ref{LemmBoundsGraph} using the fact that any
$\Sigma$-quasimodel has level at most $s+1$.
\endproof

Finally, we obtain an analogous result for pointed structures.

\begin{definition}
A {\em pointed labeled poset} is a structure $(W,\peq,\labfun,w)$ consisting of a labeled tree with a designated world $w\in W$. Given a labeled poset $\A=(W_\iA,\peq_\iA,\labfun_\iA)$ and $w\in W_\iA$, we denfine a pointed, labeled poset $\A^w = (W_\iA,\peq_\iA,\labfun_\iA, w)$. A {\em pointed simulation} between pointed labeled posets $\A=(W_\iA,\peq_\iA,\labfun_\iA,w_\iA)$ and $\B=(W_\iB,\peq_\iB,\labfun_\iB,w_\iB)$ is a simulation $\simvar\subset W_\iA\times W_\iB$ such that if $w \mathrel \simvar v$, then $w=w_\iA$ if and only if $v=w_\iB$. The notions of {\em pointed immersion,} {\em pointed condensation,} etc.~are defined analogously to Definition \ref{DefSim}.
\end{definition}

\begin{lemma}\label{LemmPointedBound}
If $\Labels$ has $n$ elements, any pointed $\Labels$-labeled poset of level at most $k$ condenses to a labeled pointed tree bounded by $\nrq{2n}{k+2}$, and there are at most $\nre{2n}{k+2}$ bimersion classes.
\end{lemma}

\proof
We may view a pointed labeled poset $\A=(W,\peq,\labfun,w)$ as a (non-pointed) labeled poset as
follows. Let $\Labels'=\Labels\times \{0,1\}$. Then, set $\labfun'(v)=(\labfun(v),0)$ if $v\ne w$,
$\labfun'(w)=(\labfun(w),1)$. Note that $\A$ may now have level $k+2$, since we may have that $u\peq w\peq v$, $\labfun(u)=\labfun(w)=\labfun(v)$, yet $\labfun'(u)\ne\labfun'(w)$ and $\labfun'(w)\ne\labfun'(v)$. By Proposition \ref{PropIsBisim}, $\A$ condenses to a generated tree $\T$ of $\G^{\Labels'}_{k+2}$ by some condensation $({\srone},{\srtwo})$. Let $w'=\srone(w)$, and consider $\T$ as a pointed structure with distinguished point $w'$. Given that $\srone$ is a surjective, label-preserving function, $w,w'$ are the only points whose label has second component $1$, and therefore $({\srone},{\srtwo})$ must be a pointed condensation, as claimed.
\endproof

\begin{proposition} \label{PropBoundPointed}
  Let $\Sigma$ be a set of formulas closed under subformulas with $|\Sigma|=s <\omega$.
  \begin{enumerate}
    \item Given a tree-like pointed $\Sigma$-quasimodel $\T$
          and a formula $\varphi$,
          there is a tree-like pointed $\Sigma$-quasimodel $\T'\simm \T$ bounded by $\nrq{2^{s+1}}{s+3}$.
          We call $\T'$ the {\em normalized pointed $\Sigma$-quasimodel for $\T$.}

    \item Given a sequence of tree-like pointed $\Sigma$-quasimodels $\T_1,\hdots,\T_n$ with $n > \nre{2^{s+1}}{s+3}$,
          there are indexes $i<j\leq n$ such that $\T_i\simm\T_j$.
  \end{enumerate}
\end{proposition}

With these tools at hand, we are ready to prove that $\iltl$ has the strong finite model property, and hence is decidable.

\section{Decidability} \label{SecDec}

The following transformations are defined
for any stratified model $\model$ and any finite, non-empty set of formulas $\Sigma$ closed
under subformulas. In each case, given a stratified model $\model=(W,{\peq},\msuccfct,V)$, we will
produce another stratified model $\model'=(W',{\peq}',\msuccfct',V')$ and a map $\pi\colon W'\to W$
such that $\Sigma_\model(w)=\Sigma_{\model'}(\pi(w))$ for all $w\in W'$.

\subparagraph*{Replace $\model_k$ with a copy of the normalized $\Sigma$-quasimodel of $\model_k$.}
Let $\T = \left(W_\T, \isuccrel_\T, \labfun_\T\right)$ be a copy of the normalized labeled tree of $\model^\Sigma_k$
such that $W_\T \cap W = \emptyset$,
and $(\srone,\srtwo)$ the condensation from $\model^\Sigma_k$ to $\T$.
The result of the transformation is the tuple $(W', \irel', \msuccfct', V')$ such that
$W' = W \cup W_\T \setminus W_k$, $\irel' = \reduc{\irel}{W \setminus W_k} \cup \left(\isuccrel_\T\right)^*$,
\[
  \msuccfct'(w) = \begin{cases}
    \srone\left(\msuccfct\left(w\right)\right) &\text{if } \msuccfct(w) \in W_k \\
    \msuccfct\left(\srtwo\left(w\right)\right) &\text{if } w \in W_\T \\
    \msuccfct(w) &\text{otherwise}
  \end{cases} \qquad
  V'(w) = \begin{cases}
    \ens{p}{p \in \labfun_T(w)} &\text{if } w \in W_\T \\
    V(w) &\text{otherwise}
  \end{cases}
\]
The map $\pi$ is the identity on $W'_i=W_i$ for $i\not=k$, and $\pi(w)=\srtwo(w)$ for $w\in W'_k$.

\subparagraph*{Replace $\model_k$ with a copy of the normalized, pointed $ \Sigma$-quasimodel of $\model_k$ preserving $w$,}
where $w \in W_k$.
The transformation is similar to the previous one except that
$\model_k$ is regarded as a pointed structure with distinguished point $w$.

\subparagraph*{Replace $\model_\ell$ with $\model_k$,}
where $k < \ell$ and there is an immersion $\simvar\colon W_k\to W_\ell$ (seen as $2^\Sigma$-labeled trees).
The result of the transformation is the tuple $(W', \irel', \msuccfct', V')$ such that $W' = W \setminus \bigcup_{k < m \le \ell} W_m,$ $\irel' = \reduc{\irel}{W'}$,
\[\msuccfct'(w) =
    \begin{cases}
      \msuccfct\left(\simvar\left(w\right)\right) &\text{if } \msuccfct(w) \in W_k \\
      \msuccfct(w) &\text{otherwise}
    \end{cases}\]
and $ V' = \reduc{V}{W'}$.

The map $\pi$ is the identity on $W'_i = W_i$ for $i < k$, on $W'_i = W_{i+\ell-k}$ for $i > k$,
and $\pi(w)=\simvar(w)$ for all $w \in W'_k$.

\subparagraph*{Replace $\model_\ell$ with $\model_k$ connecting $w_k$ to $w_\ell$,}
where $k < \ell$, $w_k \in W_k$, $w_\ell \in W_\ell$ and
there is an immersion $\simvar\colon W_k \to W_\ell$ such that $\simvar(w_k) = w_\ell$.
The transformation is defined as the previous one.

\begin{lemma} \label{lem:transformations}
  The result of any previous transformation is a stratified model such that
  $\Sigma_{\model}(w) = \Sigma_{\model'}({\pi(w)})$ for any $w\in W'$.
\end{lemma}

\begin{proof}
  The proof that $\modelbase[']$ is a model is straighforward and left to the reader.
  We prove by structural induction on $\varphi$ that
  for all transformations,
  all $w \in W'$ and all $\varphi \in \Sigma$,
  $\model', w \sat \varphi$ iff $\model, \pi(w) \sat \varphi$.
  We only detail the case for the next modality when $\model_k$ is replaced with a copy of the
  normalized $\Sigma$-quasimodel $\T$ of $\model_k$
  and $w \in W'_{k-1}$.
  The cases for the other temporal modalities are similar
  (see also the proof of Lemma~\ref{LemFinSigma}).
  The cases for the implication are similar as in the proof of Proposition~\ref{prop:IL}.
  The remaining cases are straighforward.
  Suppose that $w \in W'_{k-1}$ and $\model, \pi(w) \sat \tnext \psi$.
  Then $\psi \in \Sigma(\msuccfct(w))$.
  Since $\msuccfct'(w) = \srone(\msuccfct(w))$, $\pi(\msuccfct'(w)) = \srtwo(\msuccfct'(w))$
  and $(\srone,\srtwo)$ is a condensation,
  $\Sigma(\msuccfct(w)) = \labfun_\T(\msuccfct'(w)) = \Sigma(\pi(\msuccfct'(w)))$
  and $\model, \pi(\msuccfct'(w)) \sat \varphi$.
  By induction hypothesis, $\model', \msuccfct'(w) \sat \varphi$,
  hence $\model', w \sat \tnext \varphi$.
  The other direction is similar.
\end{proof}

Now, let us consider a stratified model $\modelbase$ with $w_0$ the root of $W_0$.
The finite model $\modelbase[^{\rm fin}]$ with a state $w_0^{\rm fin}$ such that $\wsf{w_0^{\rm fin}} = \wsf{w_0}$
is constructed by the following procedure.
This procedure is in three phases plus a final step.
At each step, the current model $\modelbase$, initialized to $\model$, is modified.
Moreover, three index variables are maintained by the procedure:
\begin{itemize}
  \item The variable $i$, initialized to $0$, indicates the current labeled trees $W_i$ which
        is considered.
  \item The variable $j$, initialy undefined, indicates the index of the first labeled trees
        occuring infinitely often up to bimersion.
  \item The variable $\ell$, initialy undefined, holds the index of the last labeled tree
        that must not be modified.
\end{itemize}
As an invariant, $\model$ is stratified until the final step
and for all $k < i$, $\model_k$ is a copy of a normalized labeled tree.

\subparagraph*{First phase.}
\begin{itemize}

\item If there is $k < i$ such that $\model_k \imm \model_i$, replace $\model_i$ with $\model_k$, set $i$ to $k+1$ and redo the same phase.

\item If not, and for all $x > i$ there is $y > x$ such that $\model_y\imm \model_i$, 
then replace $\model_i$ with a copy of its normalized $\Sigma$-quasimodel, increase $i$ by one, set $j$ and $\ell$ to $i$ and start the next phase.

\item Otherwise, replace $\model_i$ with its normalized $\Sigma$-quasimodel, increase $i$ by one and redo the same phase.

\end{itemize}

\subparagraph*{Second phase.}
In this phase, we need to care about eventualities.
To this end, a current eventuality $(w, \psi)$, initialy undefined, is maintained across the
executions of the phase. Let $w_x$ denote the element of the fulfillment of $(w, \psi)$ belonging to $W_x$ (if it exists), and $\model^+_x$ be the pointed structure $\model_x^{w_x}$.
The phase proceeds through the following steps:
\begin{itemize}
  \item If $(w,\psi)$ is defined and
        the last element of the fulfillment of $(w, \psi)$ belongs to some $W_k$ with $k \le i$
        then undefine $(w, \psi)$, set $\ell$ to $i$ and repeat the same phase.
  \item If $(w, \psi)$ is undefined
        then choose an eventuality $(w, \psi)$ such that $w \in W_j$ and
        the last element of its fulfillment belongs to some $W_k$ with $k > i$.
        If there is no such eventuality then start the next phase.
  \item If $(w, \psi)$ is defined
        and there is $k$ such that $\ell < k < i$ and $\model^+_k\imm \model^+ _i$,
        then replace $\model_i$ with $\model_k$ connecting $w_k$ to $w_i$,
        set $i$ to $k+1$ and redo the same phase.

\item        Otherwise, replace $\model_i$ with a copy of the normalized labeled tree of $\model_k$ preserving $w_i$,
        increase $i$ and redo the same phase.
\end{itemize}

\subparagraph*{Third phase.}
\begin{itemize}
        
\item If $\model_i \imm \model_j$, then start the final step.

\item If there is $k$ such that $\ell < k < i$ and $\model_k\imm\model_i$,
then replace $\model_i$ with $\model_k$,
set $i$ to $k+1$ and redo the same phase.

\item 
Otherwise, replace $\model_i$ with a copy of its normalized $
\Sigma$-quasimodel,
increase $i$ by one and redo the same phase.

\end{itemize}

\newcommand{\Qfin}{^{\rm fin}}

\subparagraph*{Final step.}
There is an immersion $\simvar \colon W_i  \to W_j$.
Construct the final tuple\linebreak $(W\Qfin, \irel\Qfin, \msuccfct\Qfin, V\Qfin)$ such that $W\Qfin = \bigcup_{0 \le m < i} W_m $, $\irel\Qfin = \reduc{\irel}{W\Qfin}$,
\[
  \msuccfct\Qfin(w) = \begin{cases}
    \simvar\left(\msuccfct\left(w\right)\right) &\text{if } w \in W_{i-1} \\
    \msuccfct(w) &\text{otherwise}
  \end{cases}
  \]
$ V\Qfin = \reduc{V}{W\Qfin}$, and $w_0\Qfin$ is the root of $W_0$ (note that $w_0\Qfin \in W\Qfin$).

\begin{lemma}\label{LemFinSigma}
  The final tuple is a model and $\wsf{w_0\Qfin} = \wsf{w_0}$.
\end{lemma}
\begin{proof}
  The proof that $\modelbase[\Qfin]$ is a model is straightforward and left to the reader.
  We prove by structural induction on $\varphi$ that
  for all $w \in W\Qfin$ and all $\varphi \in \Sigma$,
  $\model\Qfin, w \sat \varphi$ iff $\model, w \sat \varphi$.
  The cases for propositional variables and the boolean connectives are straightforward.
  The case for the next temporal modality is similar as in the proof of
  Lemma~\ref{lem:transformations}.
  For the eventually and henceforth temporal modalities,
  suppose first that $(w, \varphi)$ is an eventuality in $\model$ and $w \in W\Qfin$.
  Let $w_0 \ldots w_n$ be the fulfillment of $(w, \varphi)$ in~$\model$.
  If $w_n \in W\Qfin$ then by induction hypothesis, $(w,\varphi)$ is an eventuality in $\model\Qfin$.
  Otherwise, there is $k \le n$ such that $w_k \in W_i$.
  Therefore, $(w_k, \varphi)$ is an eventuality in $\model$ and so is $(\sigma(w_k), \varphi)$.
  Since by construction, after the second phase, 
  the length of the fulfillment of any eventuality $(v, \varphi)$ such that $v \in W_j$
  is bounded by $1 + i - j$,
  $(w, \varphi)$ is an eventuality in $\model\Qfin$.
  Conversely, suppose now that $(w, \varphi)$ is an eventuality in $\model\Qfin$ and
  let $w_0 \ldots w_n$ be its fulfillment.
  For each $k \le n$ let $m_k$ be such that $w_k \in W_{m_k}$.
  The proof is by a subinduction on the number $r$ of $k \in 1 \dts n$ such that $m_k = j$.
  If $r = 0$ then by induction hypothesis, $(w, \varphi)$ is is an eventuality in $\model$.
  If $r > 0$, let $k > 0$ be the least index such that $m_k = j$.
  If $k = n$ then suppose that $\varphi = \diam\psi$, the other case beeing symmetric.
  When have $\model\Qfin, w_k \sat \psi$ and by induction $\model, w_k \sat \psi$.
  Since $k > 0$, $w_k = S\Qfin(w_{k-1}) = \sigma(S(w_{k-1}))$ and
  since $\sigma$ is an immersion, $\model, S(w_{k-1}) \sat \psi$.
  Therefore $(w, \varphi)$ is an eventuality in $\model$.
  Finally, if $r > 0$ and $k < n$
  then $(w_k, \varphi)$ is an eventuality in $\model\Qfin$ and
  by the subinduction hypothesis $(w_k, \varphi)$ is an eventuality in $\model$.
  Since $k > 0$, $w_k = S\Qfin(w_{k-1}) = \sigma(S(w_{k-1}))$.
  Morevoer, since $\sigma$ is an immersion, $(S(w_{k-1}), \varphi)$ is an eventuality in $\model$.
  Hence $(w, \varphi)$ is an eventuality in~$\model$.
\end{proof}

\begin{lemma}\label{LemmFinalBound}
  The cardinality of $W\Qfin$ is bounded by
\[
    B(s) \eqdef \nrq{2^{s+1}}{s+3}\left(2 \nre{2^s}{s+1} + s \nrq{2^s}{s+1} \nre{2^{s+1}}{s + 3}\right)
  \]
  where $s = \card\Sigma$.
\end{lemma}
\begin{proof}
  Let us consider the stratified model $\modelbase$ obtained after the third phase.
  For all $k < i$, $W_k$ is a copy either of a normalized  $\Sigma$-quasimodel
  or of a pointed normalized $\Sigma$-quasimodel.
  By Propositions \ref{PropBound} and~\ref{PropBoundPointed},
  for all $k < i$, $\card{W_k} \le \nrq{2^{s+1}}{s+3}$.
  We prove now that
  \[i \le 2 \nre{2^s}{s+1} + s \nrq{2^s}{s+1} \nre{2^{s+1}}{s + 3}.\]
  After the first phase, by Proposition~\ref{PropBound},
  we have $j \le \nre{2^s}{s+1}$ and $\card{W_j} \le \nrq{2^s}{s+1}$.
  Therefore, during the second phase,
  the current eventuality is defined at most $s\nrq{2^s}{s+1}$ times.
  Moreover, each time the current eventuality is undefined,
  by Proposition~\ref{PropBoundPointed} we have that $i - \ell \le \nre{2^{n+1}}{n+3}$.
  Therefore, when the second phase terminates,
  \[\ell - j \le s \nrq{2^s}{s+1} \nre{2^{s+1}}{s + 3}.\]
  Finally, after the third phase, by Proposition~\ref{PropBound}, $i - \ell \le \nre{2^s}{s+1}$.
\end{proof}

We have proved the following strong finite model property.

\begin{theorem} \label{ThmFMP}
  There exists a computable function $B$ such that
  for any formula $\varphi \in \lang$,
  if $\varphi$ is satisfiable (resp. unsatisfiable)
  then $\varphi$ is satisfiable (resp. falsifiable) in a model $\modelbase$
  such that $\card W \le B(\nos\varphi)$.
\end{theorem}

\fullproof{
\begin{proof}
  In view of Theorem~\ref{TheoStrat},
  a formula $\varphi$ is satisfiable (resp. falsifiable) in a model $\model$ if and only if
  it is satisfied (resp. falsified) at the root of a stratified model $\model\Qstrat$.
  Then, by Lemma~\ref{LemFinSigma}, $\varphi$ is satisfied (resp. falsified) in $\model\Qstrat$ if and only if
  it is satisfied (res. falsified) on $(\model\Qstrat)\Qfin$,
  which is effectively bounded by $B(\nos\varphi)$ by Lemma \ref{LemmFinalBound}.
\end{proof}
}
\shortproof{}

As a corollary, we get the decidability of \iltl.

\begin{corollary}
  The satisfiability and validity problems for $\iltl$ are decidable.
\end{corollary}

\section{Conclusion} \label{SecConclusion}

We have introduced $\iltl$, an intuitionistic analogue of $\sf LTL$ based on expanding domain models from modal logic. In the literature, intuitionistic modal logic is typically interpreted over persistent models, but as we have shown this interpretation has the technical disadvantage of not enjoying the finite model property. Of course, this fact alone does not imply that $\itlb$ is undecidable, and whether the latter is true remains an  open problem. Meanwhile, our semantics are natural in the sense that we impose the minimal conditions on $\msuccfct$ so that any formula is true on an upwards-closed set under $\peq$, and a wider class of models is convenient as they can more easily be tailored for specific applications.

This is an exploratory work, being the first to consider the logic $\iltl$. As can be gathered from
the tools we have developed, understanding this logic poses many technical challenges, and many
interesting questions remain open. Perhaps the most pressing is the complexity of validity and
satisfiability: the decision procedure we have given is non-elementary, but there seems to be little
reason to assume that this is optimal. It may be possible to further `trim' the model $\model\Qfin$
to obtain one that is elementarily bounded. However, we should not expect polynomially bounded models, as $\iltl$ is conservative over intuitionistic propositional logic, which is already {\sc PSpace}-complete. Finally, we leave open the problem of finding a sound and complete axiomatization for $\iltl$.

\subsection*{Acknowledgements}

This research was partially supported by ANR-11-LABX-0040-CIMI within the program ANR-11-IDEX-0002-02.


\end{document}